\documentclass[12pt]{amsart}
\usepackage{amsmath,amssymb,amsthm,bm,longtable,color}

 \usepackage{a4wide}

\usepackage{latexsym,amssymb,color,pgf,tikz}
\usetikzlibrary{arrows,shapes,positioning}
\numberwithin{equation}{section}
\newtheorem{theorem}{Theorem}[section]

\newtheorem{lemma}[theorem]{Lemma}
\newtheorem{definition}[theorem]{Definition}
\newtheorem{remark}[theorem]{Remark}
\newcommand{\SL}{\mathop{\mathrm{SL}}}
\newcommand{\ASL}{\mathop{\mathrm{ASL}}}

\newcommand{\Sp}{\mathop{\mathrm{Sp}}}
\newcommand{\alt}{\mathop{\mathrm{Alt}}}
\newcommand{\aut}{\mathop{\mathrm{Aut}}}
\newcommand{\inn}{\mathop{\mathrm{Inn}}}
\newcommand{\soc}{\mathrm{soc}}

\newcommand{\sym}{\mathop{\mathrm{Sym}}}

\newcounter{claim}[theorem]

\title[Star normal quotients with TW type]{A new infinite family of star normal quotient graphs of twisted wreath type}

\author{Eda Kaja, Luke Morgan}

\thanks{The first author received funding
from the European Research Council (ERC) under the European
Union’s Horizon 2020 research and innovation programme
(EngageS: grant agreement No. 820148).
The second author is supported in part by the Slovenian Research Agency (research program P1-0285 and research projects J1-1691, N1-0160, J1-2451, N1-0208, J3-3001, J3-3003). We are grateful to Michael Giudici and Eric Swartz for sharing a manuscript with us.}

\address{Eda Kaja \\ TU Darmstadt, S2|15 217, Schlossgartenstrasse 7, 64289 Darmstadt}
\email{kaja@mathematik.tu-darmstadt.de}

\address{Luke Morgan \\ University of Primorska, UP FAMNIT, Glagolja\v{s}ka 8, 6000 Koper, Slovenia, and University of Primorska, UP IAM, Muzejski trg 2,  6000 Koper, Slovenia.}
\email{luke.morgan@famnit.upr.si}


\begin{document}

\begin{abstract}
We construct the first infinite families of locally arc transitive graphs with the property that the automorphism group has two orbits on vertices and is quasiprimitive on exactly one orbit, of twisted wreath type. This work contributes to Giudici, Li and Praeger's program for the classification of locally arc transitive graphs by showing that the star normal quotient twisted wreath category also contains infinitely many graphs.
\end{abstract}

 \maketitle
 
 \section{Introduction}
Let $\Gamma$ be a connected, finite, simple graph with vertex set $V\Gamma$, edge set $E\Gamma$ and automorphism group $\aut(\Gamma)$. An \textit{$s$-arc} in $\Gamma$ is a tuple $(v_0,v_1,\ldots, v_s)$ such that $v_i$ is adjacent to $v_{i+1}$ and $v_{i-1}\neq v_{i+1}$ for all $i$. 
For $G\leq \aut(\Gamma)$ the graph $\Gamma$ is \textit{locally $(G,s)$-arc transitive} if $\Gamma$ contains an $s$-arc and if  for each $v \in V\Gamma$ the stabiliser of $v$ in $G$ acts transitively on the set of $s$-arcs emanating from $v$. 
If $G$ is  transitive on $V\Gamma$, then we say $\Gamma$ is \textit{$(G,s)$-arc transitive}.
We also say that $\Gamma$ is (locally) $s$-arc transitive if it is (locally) $(\aut(\Gamma),s)$-arc transitive.
It is natural to expect that  (locally) $s$-arc transitive graphs with large values of $s$ are interesting. Weiss~\cite{weiss_1981} showed (building on work of Tutte~\cite{tutte_1947, tutte_1959} in the cubic case) that $s \leqslant 7$ for all $s$-arc transitive graphs of valency at least three.  Recently, van Bon and Stellmacher~\cite{bon_stellmacher_2015} have shown that $s\leqslant 9$ for all locally $s$-arc transitive graphs of valency at least three. 
Based on these results, one might hope to classify such graphs,  for particular values of $s$, or with particular  properties. In the vertex-transitive case, this problem is well studied; in this article we focus on the vertex-intransitive case.

 Giudici, Li and Praeger \cite{giudici_li_praeger_2003} initiated a program of global analysis aimed  towards characterising locally $(G,s)$-arc transitive graphs with $s\geq 2$  such that $G$ is vertex-intransitive. Such graphs are bipartite and the two parts of the bipartition are $G$-orbits. One of the main aspects of their program is a reduction \cite[Theorem 1.1]{giudici_li_praeger_2003} to the case where $G$ acts quasiprimitively on at least one of the two orbits, a so-called `basic' graph. Further analysis of the basic examples can be done by   utilising the O'Nan-Scott Theorem for quasiprimitive groups~\cite{praeger}, which  divides the quasiprimitive groups into eight  types  --  holomorph affine (HA), almost simple (AS), simple diagonal (SD), compound diagonal (CD), holomorph simple (HS), holomorph compound (HC), product action (PA) and twisted wreath (TW) (see Section 2 below for   details). When $G$ is quasiprimitive on both orbits \cite[Theorem 1.2]{giudici_li_praeger_2003}  shows that either the quasiprimitive type on each orbit is the same, or is a pairing of SD and PA types. All the  examples are known in the latter case (see \cite{giudici_li_praeger_2006_2}), and in the former case a classification of such graphs would imply a classification of  $ (G,s)$-arc transitive graphs  (see the discussion preceding \cite[Lemma 3.3]{giudici_li_praeger_2003}) and is therefore incomplete.
 
 When $G$ is quasiprimitive on exactly  one orbit,  \cite[Theorem 1.3]{giudici_li_praeger_2003} shows that the quasiprimitive type is HA, HS, AS, PA or TW.  A full classification was achieved in the HA, HS and AS cases. In the PA case, much is known if the action is primitive rather than quasiprimitive, but still an infinite family of such graphs is given. In the TW case however, only one example was provided.

The twisted wreath product construction was introduced by B.~H.~Neumann  \cite{neumann} and   refined by Suzuki \cite{suzuki}. Not every group that arises from the construction is primitive or even quasiprimitive -- the precise conditions are delicate; in fact  the twisted wreath groups were missed in the first version of the O'Nan-Scott Theorem for primitive groups \cite{scott}.  We refer the reader to Baddeley \cite{baddeley} for a discussion of the conditions for which a group of twisted wreath type is primitive. Furthermore,  Baddeley  \cite{baddeley2arc} has characterised those $2$-arc transitive graphs with automorphism group a quasiprimitive group of TW type. 

In this paper, we construct infinitely many locally $(G,2)$-arc transitive graphs where $G$ is vertex-intransitive and is quasiprimitive on exactly one orbit of TW type. Our main result thus shows that the last category of Giudici, Li and Praeger's global analysis program described above is also infinitely populated. In our main theorem below,  we also prove that the examples we construct are basic, with respect to the full automorphism group of the graph. This  establishes that the graphs we construct are \emph{new}, in the sense that they do not arise from one of the constructions in \cite{giudici_li_praeger_2003}, or the sequels \cite{giudici_li_praeger_2006, giudici_li_praeger_2006_hom_fact, giudici_li_praeger_2006_2} or  further studies of locally $s$-arc transitive graphs \cite{fang_li_praeger_2004, giudici_li_praeger_2005, giudici_li_praeger_2007, leemans_2009, swartz_2012}. 


 
 \begin{theorem}
 \label{mainthm}
For every prime power $q \geqslant 4$ there exists a locally $2$-arc transitive graph $\Gamma$ of valency $q^2$ such that,  for  $A=\mathrm{Aut}(\Gamma)$, the following hold:
\begin{enumerate}
\item $A$ has two orbits $\Delta_1$ and $\Delta_2$  on $V\Gamma$;
\item the action of $A$ on $\Delta_1$ is quasiprimitive of twisted wreath type;
\item there is a normal subgroup $Y$ of $A$ that is intransitive on $\Delta_2$ and the quotient graph $\Gamma_Y$ is the complete bipartite graph $K_{1,q^2}$.
\end{enumerate}
 \end{theorem}
 
  A complete bipartite graph $K_{1,q^2}$ is usually called a \textit{star}. In part (iii), the vertex set of the  \textit{quotient graph} $\Gamma_Y$ is the set of $Y$-orbits in $V\Gamma$ and there is an edge in $\Gamma_Y$  between two vertices if there is an edge between the corresponding $Y$-orbits  in $\Gamma$.
It is these properties that give rise to the name `star normal quotient' to describe  graphs such as those appearing in Theorem~\ref{mainthm}.

%
%

In Section 2 we assemble some results on overgroups of primitive and quasiprimitive groups. The construction of the graphs appearing in the theorem takes place in Section 3, and in Section 4 we analyse the automorphism groups of the constructed graphs.

%
%
%
%
%
 \section{Preliminaries}
 
The construction of the graphs appearing in Theorem~\ref{mainthm} is based on the   \emph{coset graph} construction.

\begin{definition}
Let $G$ be a group and let $L,R < G$ be such that $L\cap R$ is core-free in $G$. Let $\Delta_1=\{Lx:x\in G\}$ and let $\Delta_2=\{Ry:y\in G\}$. Define the bipartite graph $\Gamma=\mathrm{Cos}(G,L,R)$ such that $V\Gamma=\Delta_1\cup \Delta_2$ and $Lx\sim Ry$ if and only if $Lx \cap Ry \neq \emptyset$. We refer to $(L,R,L\cap R)$ as the associated \textit{amalgam}.
\end{definition}


The following lemma shows that every locally $s$-arc transitive graph arises from the above construction.

\begin{lemma}[{\cite[Lemma 3.7]{giudici_li_praeger_2003}}]
    For a group $G$ and subgroups $L,R < G$ such that $L\cap R$ is core-free in $G$, the graph $\Gamma=\mathrm{Cos}(G,L,R)$ satisfies the following properties:
    \begin{enumerate}
        \item $\Gamma$ is connected if and only if $\langle L,R \rangle  = G$;
        \item $G\leq \mathrm{Aut}(\Gamma)$ and $\Gamma$ is $G$-edge transitive and $G$-vertex intransitive;
        \item $G$ acts faithfully on both $\Delta_1$ and $\Delta_2$ if and only if both $L$ and $R$ are core-free.
    \end{enumerate}
    Conversely, if $\Gamma$ is $G$-edge transitive and not $G$-vertex transitive, and $v$ and $w$ are adjacent vertices, then $\Gamma \cong \mathrm{Cos}(G,G_v,G_w)$.
\end{lemma}

A transitive permutation group $G$ on a finite set $\Omega$ is \textit{quasiprimitive} if each non-trivial normal subgroup of $G$ acts transitively on $\Omega$.  Praeger~\cite{praeger} classified quasiprimitive groups in   an O'Nan-Scott type theorem and separated them into eight types \cite{praeger_1997} (see also~\cite[Section 2]{giudici_li_praeger_2003}).
The classification is based around the abstract structure and the action of the \emph{socle} (the product of the minimal normal subgroups). If $G$ is a finite quasiprimitive permutation group with socle $X$, then either $X$ itself is a minimal normal subgroup, or $X=X_1X_2$ for some minimal normal subgroups $X_1$ and $X_2$ of $G$. In the latter case,  $X_1 \cong X_2 \cong T^k$ for some finite non-abelian simple group $T$ and integer $k$; if $k=1$ then $G$ is holomorph simple (HS) type and if $k>1$ then $G$ is  holomorph compound (HC) type. If $X$ itself is a minimal normal subgroup of $G$, then either $X $ is abelian and regular and $G$ is holomorph affine (HA) type, or, $X\cong T^k$ for some non-abelian simple group $T$. If $k=1$ then $G$ is almost simple (AS) type. For  $k>1$, the type depends upon the structure of a point stabiliser, $H$ say, in $X$. If $H=1$ then $X$ is regular and $G$ is of twisted wreath (TW) type.  If $H$ is non-trivial and  projects to a proper subgroup in each of the simple direct factors of $X$ then $G$ has product action (PA) type. If $H$ is non-trivial and projects onto each of the simple direct factors, then $G$ is simple diagonal (SD) type  if $H\cong T$ and compound diagonal (CD) type if $H \cong T^\ell$ for some integer $\ell>1$.

We will require the following results about the inclusion problem of primitive and quasiprimitive groups. Recall that a   permutation group $G$ acting on a set $\Omega$ is primitive if the only partitions of  $\Omega$ that are preserved by $G$ are the partition of $\Omega$ into singletons and the partition $\{ \{ \Omega \}\}$. For any $\omega \in \Omega$, there is a one-to-one correspondence between the overgroups of  the point stabiliser $G_\omega$ of $\omega$ in $G$ and the partitions of $\Omega$ that are preserved by $G$. Thus a transitive permutation group is primitive if and only if point stabilisers are maximal subgroups.

\begin{lemma}[{\cite[Proposition 8.1]{praeger_1990}}]
\label{lem: overgroup of sd}
Suppose that $G \leqslant H \leqslant \sym(n)$ and that $G$ is  primitive. If $G$ has type SD, then one of the following hold: $H=\alt(n)$; $H=\sym(n)$;  $\soc(G)=\soc(H)$ and $H$ is of SD type.
\end{lemma}

\begin{lemma}
Suppose that $G \leqslant H \leqslant \sym(n)$ and that $G$ is  quasiprimitive. If $G $ has type TW and if $H$ is primitive, then one of the following hold:
\label{lem: prim overgroup of tw}
\begin{enumerate}
\item $H=\alt(n)$ or $\sym(n)$;
\item $H$   is of type PA, HC, SD or CD;
\item $H$ is of type TW and $\soc(G)=\soc(H)$.
\end{enumerate}
\end{lemma}
\begin{proof}
This follows from \cite[Theorem 1.2]{baddeley_praeger_2003}
\end{proof}

If $G$ is a permutation group acting on a set $\Omega$ and $G$ preserves a partition $\mathcal B$ of $\Omega$, then $G^{\mathcal B}$ denotes the permutation group induced on $\mathcal B$ by $G$. 

\begin{lemma}
\label{lem: imprim overgroup of tw}
Suppose that $G \leqslant H \leqslant \sym(n)$ and that $G$ is quasiprimitive. If $G$ has type TW and $H$ is imprimitive, then for  partition  $\mathcal B$  that is preserved by $H$, one of the following hold:
\begin{enumerate}
\item  $H $ has type TW and  $\soc(G)=\soc(H)$,
\item $\soc(G)=T^k < \soc(H) = S^k$  for non-abelian simple groups $S$, $T$, and both $G^{\mathcal B}$  and  $H^{\mathcal B}$ are of type PA.
\end{enumerate}
\end{lemma}
\begin{proof} This follows from \cite[Theorem 2]{praeger_qp_inclusions_2003}.
\end{proof}

If $G$ is a transitive permutation group acting on the set $\Omega$ and $Y$ is an intransitive normal subgroup of $G$, then the set of orbits of $Y$ on $\Omega$, denoted by $\Omega/Y$ forms a partition of $\Omega$ that is preserved by $G$.

 \section{A construction}
 
 Hypothesis: for some prime $p$ the following hold:
 \begin{enumerate}
 \item $P = V \rtimes Q$ is an affine $2$-transitive group of degree a power of $p$;
 \item  the stabiliser $Q_x$ of a non-zero vector $x \in V$,  has a non-trivial centre with order divisible by $p$;
 \item there is a non-abelian finite simple group $T$ and a homomorphism $\phi : Q \rightarrow \aut(T)$ such that $\phi(Q)$ contains $\mathrm{Inn}(T)$.
 \end{enumerate}
Note that the groups $q^2 : \SL(2,q)$ for a prime power $q\geqslant 4$, $ q^4 : \Sp(4,q)$ and  $3^6 : \SL(2,13)$  satisfy the hypothesis (see \cite[Table 7.3]{cameron}).
 

Let $k=|V|$. 
Define $G = T \mathrm{twr}_\phi P$, and identify $P$ and $Q$ as subgroups of $G$. Let $N$ be the base group of $G$ so that 
$$N = \{ f : P\rightarrow T \mid f(hq) = f(h)^{\phi(q)}\text{ for all } h\in P, \, q\in Q \}.$$
For $g\in P$, the action of $h$ on $N$ is 
$$ f^g (p) = f (gp).$$
Fix a left transversal $z_1,\ldots,z_k$ to $Q$ in $P$ so that each $f\in N$ is uniquely determined by the images $f(z_1)$, $f(z_2)$, \ldots, $f(z_k)$. In this way, we have $N \cong T^k$. We let $T_i$ be the subgroup of $N$ consisting of the functions that map  $z_j$ to the identity of $T$ for all $j\neq i$ and let $N_ i = \prod_{j \neq i}T_i$. The action of $G$ by conjugation on the set $\{T_1,\ldots,T_k\}$ is permutationally equivalent to the action of of $P $ on the coset space $[P:Q]$. With respect to this, we   have $Q= N_P(T_k) = N_P( N_k)$.

Since $V$ is a complement to $Q$ in $P$, we may take the elements of $V$ to be a left transversal of $Q$ in $P$. We record the following observation: if $q\in Q$ is such that $qz_i = z_i q$, then for $f\in T_i$ we have
$$f^q(z_i) = f(qz_i) = f(z_iq) = f(z_i)^{\phi(q)}$$
so that $q$ induces on $T_i\cong T$ the automorphism $\phi(q)$.

We view $G$ as a permutation group on the set $[G:P]$. Since $\mathrm{core}_P(Q) = 1$, $G$ acts faithfully on $[G:P]$ and $G$ is a quasiprimitive group of twisted wreath type (see \cite[Section 2]{praeger}).

\begin{lemma}
 \label{lem: action of G on D1}
 The group $G$, as a permutation group on the set $[G:P]$, is quasiprimitive of TW type. The action is imprimitive, and $G$ preserves a unique (non-trivial) partition $\Pi$ which corresponds to the overgroup $PC_N( V \ker(\phi))$ of $P$. The induced action of $G$ on $\Pi$ is primitive of type SD and for $\pi \in \Pi$ the action of $G_\pi$ on $\pi$ is primitive of type HS.
 \end{lemma}
 \begin{proof}
 First we gain some insight on the possible partitions preserved by $G$. Since $G=NP$, if $P \leqslant H < G$, then $H = P ( H \cap N)$. Set $M=H \cap N$ and note that  $M$  is normalised by $P$. Since $P$ is transitive on the $k$ simple direct factors of $N$, the projections of $M$ to the simple direct factors are isomorphic. Further, since the projection of $M$ to $T_k$ is normalised by $Q$ and $Q$ induces $\mathrm{Inn}(T_k)$ on $T_k$ by conjugation,  each projection is either trivial (and hence $M=1$) or each projection is surjective. 
  It follows that $M$ is a subdiagonal subgroup of $N$, and therefore by Scott's Lemma (see \cite[Theorem 4.16]{PraegerSchneider}) that $M$ is a strip, that is, $M=  \prod_{i\in I} D_i$ where $I$ forms a partition of $\{1,\ldots,k\}$ and each $D_i$ is a diagonal subgroup of $\prod_{ j\in I_i} T_j$ where $I = I_1 \cup \ldots \cup I_{|I|}$. Clearly the set $I$ forms a partition preserved by $G$. Since $G$ is primitive on the set $\{T_1,\dots,T_k\}$ of simple direct factors of $N$, $I$ is a trivial partition. Since $M$ is a proper subgroup of $N$, we have $|I|=1$ and thus $M$ is a full diagonal subgroup of $N$. Hence $M\cong T$. Since $P$ normalises $M$, we have a homomorphism $c : P \rightarrow \aut(T)$ which is the map induced by conjugation. Since $Q$ induces $\mathrm{Inn}(T)$ on $T_k\cong T$ we see that $c(Q) = \mathrm{Inn}(M)$. This means $\inn(M)$  normalises $c(V)$, and since $c(V)$ is elementary abelian, the only possibility is that $c(V)=1$. Hence $V \leqslant C_P(M)$ and $C_P(M) = V C_Q(M)$. Again, since $Q$ must induce $\mathrm{Inn}(T)$ on $M \cong T$, we have that $C_Q(M)=\ker(\phi) = Z(Q)$ so that $C_P(M)=V \ker(\phi)$ and $M \leqslant C_N(V \ker(\phi))$.
  %
  %
  %
  Thus we have proved that the only overgroups of $P$ correspond to subgroups of $C_N(V\ker(\phi))$ that are normalised by $Q$. We proceed to analyse the action of $Q$ on $C_N(V\ker(\phi))$ to find  all partitions that are preserved by $G$.
 
Since $Q$ is maximal in $P$, the map $\phi :Q \rightarrow \aut(T)$ can be extended (uniquely)  to the whole of $P$, and this extension is a homomorphism $\hat{\phi}$ with kernel $V\ker(\phi)$.  The twisted wreath product $ T \mathrm{twr}_{\hat{\phi}} P$ shows that $G$ is imprimitive, since the subgroup
$$\hat{N} = \{ f : P \rightarrow T \mid f(hq) = f(h)^{\hat{\phi}(q)} \text{ for all } h,q\in P \}$$
is a subgroup of $N$ that is normalised by $P$. Further, $\hat{N} \cong T^\ell$ where $\ell=|P:P|=1$. The subgroup $\hat{N}$ can be identified as the ``first-coordinate'' subgroup of the twisted wreath product  $T \mathrm{twr}_{\hat{\phi}} P$, so that  $\hat{N} = \{ f_t : t\in T\}$ where $f_t : P \rightarrow T$ is defined by
$$f_t (h) = t^{\hat{\phi}(h)}$$
and in this way, for any $r \in V\ker(\phi) $ and for any $h\in P$ we have 
$$(f_t)^r (h) = f_t(rh) = t^{\hat{\phi}(rh)} = t^{\hat{\phi}(h)} = f_t(h)$$
so that $V\ker(\phi)$ centralises $\hat{N}$. If $q\in Q \setminus \ker(\phi)$, then there is $t\in T$ such that $t^{\phi( q)} \neq t$. Thus 
$$(f_t)^q (1) = f_t(q) = t^{\hat{\phi}(q)}=t^{\phi(q)} \neq t = f_t(1)$$
so that $Q$ acts non-trivially on $\hat{N}$. Hence $\hat{N} = C_N( V\ker(\phi) )$. In particular, since $Q$ normalises no proper non-trivial subgroup of $C_N(V\ker(\phi))$, the only possible overgroups of $P$ in $G$ are $C_N(V\ker(\phi))P$ and $G$ itself.

Let $\Pi$ denote the  partition  that is given by the overgroup $\hat{N} P$ and let $\pi$ denote the orbit of $\hat{N}$ so that $G_\pi=\hat{N}P$. Since $G$ is quasiprimitive, the action of $G$  on $\Pi$ is faithful and quasiprimitive. Further, since $N_\pi = \hat{N}$ is a full diagonal subgroup of $N$, the action is of SD type. The action of $G$ on the set of simple direct factors of $N$ is equivalent to the action of $P $ on the set of  vectors of $V$, and is therefore primitive. Hence $G^{\Pi}$ is a primitive group of SD type.


Finally, consider  the action of $G_\pi = \hat{N}P$ on $\pi$. The kernel of the action is the largest subgroup of $P$ normalised by $\hat{N}$, and such a subgroup must commute with $\hat{N}$. From above, we have that the kernel of the action is  $V\ker(\phi)$. Thus 
$$G_\pi^\pi = \hat{N} \rtimes P / (V \ker(\phi)) = \hat{N} \rtimes \phi(Q) \cong T \rtimes \mathrm{Inn}(T)$$
and since $N_\pi^\pi \cong T $ is a normal subgroup,  $G_\pi^\pi$ is primitive of type HS.
%
 \end{proof}

\begin{lemma}
There exists a subgroup $R$ of $N_{k}$ of order $|V|$ that is normalised by $Q$, and $RQ \cong P$.
\end{lemma}
\begin{proof}
Recall that $Q=N_P(T_k)=N_P(N_k)$ and for $i=1,\ldots,k-1$, set $Q_i = N_Q(T_i)$. By our hypothesis, $Q_i$ has non-trivial centre with order divisible by $p$. Further, $Q_1$ induces inner automorphisms on $T_1$ corresponding to the image of $Q_1$ under the map $\phi$, and therefore normalises a subgroup $U_1$ of order $p$. For $g\in Q$ such that $(T_1)^g=T_i$, let $U_i = (U_1)^g$ (and note that the definition of $U_i$ is independent of the choice of $g$ since $Q_1=N_Q(T_1)$). Let $W = \langle U_1,\ldots,U_{k-1} \rangle$. Since $U_i \leqslant T_i$, we have $[U_i,U_j]=1$ for all $i,j$. Thus $W$ is an elementary abelian group of order $p^{k-1}$ and $W$ is normalised by $Q$. Further, since $U_1$ is the trivial module for $Q_1$, $W$ is the permutation module for $Q$ of dimension  $k-1$ and there is a basis for $W$ which $Q$ permutes as it does the set of non-zero vectors of $V$.

(Aside: to see this directly, pick $u_1 \in U_1$ so that $U_1 = \langle u_1\rangle$. Then set $u_i = (u_1)^g$ if $g\in Q$ is such that $(T_1)^g = T_i$. Let $z_1,\ldots,z_{k-1}$ be a right transversal to $Q_1$ in $Q$, then for $g\in Q$ there is $r\in Q_1$ and $i$ such that $g=rz_i$, then
$$(u_1)^g = (u_1)^{rz_i} = u_1^{z_i} = u_i$$
and hence for any $j$, there is $h\in Q$ such that 
$$(u_j)^g = ((u_1)^h)^g = (u_1)^{hg} = u_s$$
for some $s$. Hence  $Q$ simply permutes the elements $\{u_1,\ldots,u_{k-1}\}$ which form a generating set for $W$, and therefore $W$ is simply the permutation module for $Q $ with a basis corresponding to the non-zero vectors of $V$.)


Since  $Q_1$ fixes the vector $x \in V$, it preserves the $1$-dimensional subspace $\langle  x \rangle$ of $V$ that is isomorphic to $U_1$. Thus there is an injective $Q_1$-module homomorphism $f: U_1 \rightarrow V$. By \cite[Frobenius Reciprocity Theorem, pg.~165]{alperin_bell_1995}, there is a $Q$-module homomorphism $F: W \rightarrow V$ such that $F$ extends $f$. In particular, $F$ is also non-zero, and therefore $F$ is surjective since $V$ is an irreducible $Q$-module. Thus there is a submodule $Y$ of $W$ such that $W/Y \cong V$.  Now $W$ is self-dual since the permutation module supports a $Q$-invariant symmetric bilinear form. Hence there is  a submodule, $R$  say, of $W$ such that $R \cong V$ as a $Q$-module. Hence $Q$ acts on $R$ as it does on $V$, and $RQ \cong P$.
%
%
\end{proof}

\begin{remark}
Our hypothesis does not hold for $G= \ASL(5,2)$. In this case, for a non-zero vector $x$ we have  $G_x \cong 2^4 : \SL(4,2)$, so that $G_x$ has trivial centre. If $W$ is the module for $G$ induced from the natural module of $G_x/O_2(G_x) = \SL(4,2)$, then computations in {\sc Magma}~\cite{MR1484478} show that there is no submodule of $W$ of dimension $5$.
\end{remark}

Let $\Gamma$ be the coset graph $\mathrm{Cos}(G,P,RQ)$. Let $u$ and $v$ denote  the cosets $P$ and $RQ$ respectively, so that $G_u = P$ and $G_v = RQ$. Let $\Delta_1=u^G$ and $\Delta_2 = v^G$.
 
\begin{lemma}\label{lem: gamma connected}
The graph $\Gamma$ is connected and of valency $k=|V|$. The group $G$ acts faithfully on the two orbits, $\Delta_1$ and $\Delta_2$ and is locally $2$-arc transitive.  The group $G^{\Delta_1}$  is quasiprimitive of TW type. The group $G^{\Delta_2}$ is not quasiprimitive; $N$ is intransitive with $k$ orbits.
\end{lemma} 
 \begin{proof}
 By Lemma~\ref{lem: action of G on D1}, $PC_N(V \ker(\phi))$ is the unique overgroup of $P$ in $G$, and since $Q$ induces the inner automorphism group on $C_N(V \ker(\phi)))\cong T$, we have that   $R \nleqslant C_N(V \ker(\phi))$. Thus  $G = \langle P, RQ \rangle$ and $\Gamma$ is connected. From the maximality of $Q$ in $P$ and in $RQ$, we have  $Q = P \cap RQ $. It follows that $\Gamma$ has valency $k$. Clearly $P$ and $RQ$ have trivial core in $G$, and therefore $G$ is faithful on both $\Delta_1$ and $\Delta_2$. The actions of $P$ and $RQ$ on the sets $[P:Q]$ and $[RQ:Q]$ are both  equivalent to the $2$-transitive action of $P$ on  $k$ points and hence $G$ is locally $2$-transitive. The action of $G$ on $\Delta_1$ is as the  TW group first constructed above. Since $RQ \leqslant NQ < G$, the action of $G$ on $\Delta_2$ is not quasiprimitive, and $N$ has $|G:NQ|=|P:Q|=k$ orbits on $\Delta_2$.
 \end{proof}

 

%
%
%
%
%
%
 
 \section{The automorphism group of $\Gamma$}
 
 We continue with the notation from the previous section.
Since $\Gamma$ is bipartite, we let $A$ be the normal subgroup of $\aut(\Gamma)$ that preserves the bipartition which has index at most two in $\aut(\Gamma)$. We will establish that the actions of $A$ on the two parts of $\Gamma$ are not equivalent, and this will show that $A=\aut(\Gamma)$. Let $s\in \mathbb N$ be such that $A$ is locally $s$-transitive. Since $\Gamma$ is locally $(G,2)$-arc transitive, we have that $s\geqslant 2$.

\begin{lemma}
\label{lem:not alt}
If $1\neq X$ is a normal subgroup of $A$ that is intransitive on   $\Delta_1$, then the set of orbits of $X$ on $\Delta_1$ is $\Pi$ and  $A^\Pi$ is primitive. Further, if $X$ is intransitive on $\Delta_2$, then $A^\Pi$ does not contain $\alt(\Pi)$.
\end{lemma}
\begin{proof}
Lemma~\ref{lem: A qp on D1} shows that  $\Pi$ is the  unique  partition of $\Delta_1$ that is   preserved by $G$, hence the set of  orbits of $X$ is $\Pi$ and   $A^\Pi$ is primitive.  Now assume that $X$ is intransitive on $\Delta_2$. It follows that $X$ is semiregular on both $\Delta_1$ and $\Delta_2$. Further, we have $|X|=|\pi|$ and  $X $ is a normal subgroup of $J= \langle X,  \hat{N} , P\rangle $. By Lemma~\ref{lem: A qp on D1}  $\hat{N} P=G_\pi$ acts on $\pi$ as a primitive group of HS type, so $J$ is also primitive on $\pi$ and  \cite[Proposition 8.1]{praeger_1990}  shows that $J^\pi$ is $\alt(\pi)$, $\sym(\pi)$, has type SD or type HS. Of the possibilities,  only HS groups have normal subgroups of order $|X|=|T|$, so therefore $J^\pi$ is primitive of type HS and we have $X \cong T$. 

Suppose, for a contradiction, that $A/X = A^{\Pi}$ contains $\alt(\Pi)$. We claim that $C_A(X)$ contains a normal subgroup isomorphic to $\alt(\Pi)$.  Since $C_A(X) \cap X = Z(X)=1$, we have that $C_A(X) \cong C_A(X) X / X \unlhd A/X$. Hence the claim is true, unless $C_A(X)=1$. In this case, then $A$ embeds into $\aut(T)$.  The order of $\aut(T)$ divides $|T|!$, and yet $|\alt(\Pi)|=(|T|^{k-1})!/2$ divides $|A|$, which is a contradiction since $k>2$. 
Hence $C_A(X)$ contains a normal subgroup of index at most two that is isomorphic to $\alt(\Pi)$, thus $ [C_A(X),C_A(X)]\cong \alt(\Pi)$. Let $H=X [C_A(X),C_A(X)]$ and note that $|A:H| \leqslant 2$. Hence $|G : G\cap H| \leqslant 2$. Thus $|G_v : G_v \cap H| \leqslant 2$ and $|G_u : G_u \cap H| \leqslant 2$. In particular, both $G_v \cap H_v = G_v \cap H$ and $G_u \cap H_u=G_u \cap H$ are normal subgroups of $G_v$ and $G_u$ respectively. Since the valency of $\Gamma$ is at least three, and since $G_v$ and $G_u$ act $2$-transitively on $\Gamma(v)$ and $\Gamma(u)$ respectively, it must be that $G_v \cap H$ and $G_u \cap H$ are transitive on $\Gamma(v)$ and $\Gamma(u)$, respectively. Hence $H$ is locally transitive, so $H = \langle H_u, H_v \rangle$. Order considerations yield:
$$|H| = |X| | \alt(\Pi)| = |T| \frac{(|T|^{p^2-1})!}{2}$$
and $$|H_u| = |H_v | = \frac{|H|}{|T|^{p^2}} = \frac{|T| (|T|^{p^2-1})!}{2|T|^{p^2}} =  \frac{ (|T|^{p^2-1})!}{2|T|^{p^2-1}} = \frac{(|T|^{p^2-1}-1)!}{2} = |\alt(|\Pi|-1)|$$

Now $H_v \cap X=1=H_u \cap X$ since $X$ is semiregular on both $\Delta_1$ and $\Delta_2$, so we have that both $H_v$ and $H_u$ are isomorphic to subgroups of $\alt(\Pi)$ with the same order as a point stabiliser in $\alt(\Pi)$. It follows that $H_v \cong H_u \cong \alt(|\Pi|-1)$. In particular, both $H_v$ and $H_u$ are simple. If $H_v \cap C_H(X)$ is trivial, then $H_v$ is isomorphic to a subgroup of $H/C_H(X) \cong T$, which as mentioned above, is impossible. Hence $H_v \leqslant C_H(X)$. Similarly $H_u \leqslant C_H(X)$ and therefore $H \leqslant C_H(X)$, which means $H = C_H(X)$. This gives  $ X = X \cap H = X \cap C_H(X) = Z(X) = 1$, a contradiction.
\end{proof}

\begin{lemma}\label{lem: uns intransitive on D12}
There is a unique normal subgroup of $A$ that is intransitive on both $\Delta_1$ and $\Delta_2$.
\end{lemma}
\begin{proof}
Note that the identity subgroup satisfies the conclusion of the theorem, so we may assume for a contradiction that there is a non-trivial normal subgroup $X$ of $A$ that is intransitive on $\Delta_1$ and $\Delta_2$. Further, after replacing $X$ with an overgroup, we may assume  that $X$ is maximal with this property. Hence we may apply \cite[Theorem 1]{giudici_li_praeger_2003} and consider the outcomes (i)--(iii) separately. 

In the  case (i), we have $\Gamma_X = K_{k,k}$. Since $G$ acts quasiprimitively on $\Delta_1$, $G$ is isomorphic to its image in $\aut(\Gamma_X)$. This implies that  $|T|^{k}$  divides  $|\sym(k)|$, and a contradiction is obtained by considering the power of $p$ that divides each group.

In the  case (ii),  we have that $A^\Pi$ is an overgroup of a primitive group of SD type. By Lemma~\ref{lem:not alt} $A^\Pi$ does not contain the alternating group on $\Pi$. Hence by Lemma~\ref{lem: overgroup of sd}  $A^\Pi$ must be of type SD. By \cite[Theorem 1.3]{giudici_li_praeger_2003} it must be that $A^{\Delta_2/X}$ is of type PA. According to \cite[Theorem 1.2]{giudici_li_praeger_2006_2}, $\Gamma_X$ arises from either \cite[Construction 3.3]{giudici_li_praeger_2006_2} or from a normal quotient of such a graph as in \cite[Construction 3.10]{giudici_li_praeger_2006_2}. Since all such graphs are not  regular, we have a contradiction.

In case (iii), $A$ is quasiprimitive on exactly one of $\Delta_1/X = \Pi$ or $\Delta_2/X$. Since  $G$ is primitive on $\Pi$, $A^\Pi$ is primitive (and therefore quasiprimitive) on $\Pi$. Hence $A$ is not quasiprimitive on $\Delta_2 / X$. By \cite[Theorem 1.3]{giudici_li_praeger_2003}, the quasiprimitive type of $A^\Pi$ is either HA, HS, AS, PA or TW. Since $G^\Pi$ is primitive of SD type, Lemma~\ref{lem: overgroup of sd} implies that  $A^\Pi$ has type AS and $A^\Pi$ contains $\alt(\Pi)$. Lemma~\ref{lem:not alt} delivers a contradiction.
\end{proof}

\begin{lemma}
\label{lem: A qp on D1}
There is no normal subgroup  of $A$ that is intransitive on $\Delta_1$.
\end{lemma}
\begin{proof}
Assume for a contradiction that $Y$ is a normal subgroup  of $A$ that is intransitive on $\Delta_1$. By the previous lemma, $Y$ must be transitive on $\Delta_2$. By Lemma~\ref{lem: action of G on D1}, the set of orbits of $Y$ on $\Delta_1$ must be the set $\Pi$, and therefore $|u^Y| = |T|$.   Since each vertex in $u^Y$ has $k$ neighbours in $\Delta_2$,  there are exactly $k|T|$ edges between $\pi$ and $\Delta_2$. On the other hand,  $\{u^y,v^y\}$ is an edge between $u^Y$ and $\Delta_2$ for all $y\in Y$.  Since $Y$ is transitive on $\Delta_2$, there are at least $|\Delta_2|= |T|^k$ edges between $\Delta_2$ and $u^Y$, a contradiction.
\end{proof}

\begin{lemma}\label{lem: A qp only on D1}
The action of $A$ on $\Delta_2$ is not quasiprimitive and the action of $A$ on $\Delta_1$ is quasiprimitive of type TW.
\end{lemma}
\begin{proof}
Lemma~\ref{lem: A qp on D1} allows us to consider the inclusion $G^{\Delta_1} \leqslant A^{\Delta_1}$ of quasiprimitive groups. We consider two cases according to whether  $A^{\Delta_1}$ is primitive or not.

\textbf{Case 1:} $A^{\Delta_1}$ is primitive.

Since $s\geqslant 2$, $A^{\Delta_1}$ cannot be of type HC or CD by \cite[Theorem 1.2]{giudici_li_praeger_2003}. Hence Lemma~\ref{lem: prim overgroup of tw} implies that $A^{\Delta_1}$ contains the alternating group on $\Delta_1$, has type TW or has type PA. Since $\Delta_1$ and $\Delta_2$ have the same cardinality, if $A^{\Delta_1}$ contains the alternating group, then $A^{\Delta_2}$ also contains the alternating group, and $\Gamma$ is the complete bipartite graph or the empty graph, a contradiction.   Hence we may assume that $A^{\Delta_1}$ is  primitive of type PA or TW. 

Suppose that $A^{\Delta_2}$ is quasiprimitive. Since $\Gamma$ is regular, \cite[Theorem 1.2]{giudici_li_praeger_2006_2} shows that the type  of $A^{\Delta_2}$ must be the same type as $A^{\Delta_1}$.  If $A^{\Delta_1}$ is TW, then Lemma~\ref{lem: prim overgroup of tw} shows that the socle of $A^{\Delta_1}$ is equal to the socle of $G^{\Delta_1}$. Since both groups act faithfully on $\Delta_1$, this implies that $\soc(G)=\soc(A)$ must act transitively on $\Delta_2$, since  $A^{\Delta_2}$ is quasiprimitive. Hence $\soc(A)=\soc(G)$  acts regularly on $\Delta_2$, a contradiction to Lemma~\ref{lem: action of G on D1} which says $\soc(G)$ is intransitive  on $\Delta_2$. Hence the type of $A$ (on both $\Delta_1$ and $\Delta_2$) is PA. Since $s\geqslant 2$ and since $A^{\Delta_1}$ is primitive, we obtain a contradiction from \cite[Theorem 2.1]{EricMichael}.

We may now assume that $A^{\Delta_2}$ is not quasiprimitive. If $A^{\Delta_1}$ is of type PA, then  \cite[Theorem 1.2]{giudici_li_praeger_2006} shows that $\Gamma$ is  the vertex-maxclique incidence graph of the Hamming graph $H(\ell,n)$ for some integers $\ell$ and $n$.  Since the vertex-maxclique incidence graph of $H(\ell,n)$ has valencies $\ell$ and $n$,  we have $\ell=n=k$. Hence the number of vertices in $\Gamma$ is $2 n^{\ell-1}\ell = 2 n^n = 2 k^{2k}$, a contradiction since $k$ is a prime power. Hence   $A^{\Delta_1}$ is primitive of TW type and $A^{\Delta_2}$ is not quasiprimitive.

\textbf{Case 2:}  $A^{\Delta_1}$ is imprimitive, and therefore preserves partition  $\Pi$. Since $G^{\Pi}$ is primitive of SD type by Lemma~\ref{lem: action of G on D1}, Lemma~\ref{lem: overgroup of sd} shows $A^{\Pi}$ cannot have type PA. Hence Lemma~\ref{lem: imprim overgroup of tw} (1) must hold, so that  $\soc(G)=\soc(A)$ and $A^{\Delta_1}$ has type TW. Further, since $\soc(G)=\soc(A)$, $A^{\Delta_2}$ is  not quasiprimitive. 
\end{proof}

\section{The proof of Theorem~\ref{mainthm}.}
Let $P$ be one of the groups $q^2:\SL(2,q)$, $q^4:\Sp(4,q)$ or $3^6:\SL(2,13)$ (viewed as $2$-transitive groups) which satisfy the hypothesis of Section~3. 
Let $G$ be the twisted wreath product constructed in Section~3 and let $\Gamma$ be the graph constructed in Lemma~\ref{lem: gamma connected}. 
From Lemma~\ref{lem: gamma connected} we have that $\Gamma$ is a connected graph of valency $q_0^2$ where $q_0=q$, $q_0=q^2$ or $q_0=3^3$ in the respective cases. 
Let $\Delta_1$ and $\Delta_2$ denote the two parts of $\Gamma$ and let $A$ be the subgroup of $\aut(\Gamma)$ fixing the two parts. Lemma~\ref{lem: A qp only on D1} shows that $A$ is not quasiprimitive on $\Delta_2$, and since $A$ has index at most two in $\aut(\Gamma)$, this proves $A=\aut(\Gamma)$ and establishes part (1) of Theorem~\ref{mainthm}. Lemma~\ref{lem: A qp on D1} shows that part (2) of Theorem~\ref{mainthm} holds. Finally,  Lemma~\ref{lem: A qp only on D1} shows that $A$ is not quasiprimitive on $\Delta_2$, and so there exists an intransitive normal subgroup $Y$ of $A$. Since the orbits of $Y$ are preserved by $G$, Lemma~\ref{lem: action of G on D1} shows the orbits form the partition $\Pi$, and hence $\Gamma_Y \cong K_{1,q^2}$. Thus part (3) of Theorem~\ref{mainthm} holds and the proof of the theorem is complete.


\bibliographystyle{siam}
\bibliography{bibliography}

\end{document}